\documentclass[oneside]{amsart}

\usepackage[scale=0.7]{geometry}
\usepackage[colorlinks=true,urlcolor=blue,citecolor=blue]{hyperref}
\usepackage{amssymb}

\newcommand{\cei}[1]{\left\lceil #1 \right\rceil}
\newcommand{\abs}[1]{\left\lvert #1 \right\rvert}
\newcommand{\F}{\mathcal F}
\newcommand{\G}{\mathcal G}
\newcommand{\C}{\mathcal C}
\newcommand{\D}{\mathcal D}

\newtheorem{theorem}{Theorem}[section]

\newtheorem{lemma}[theorem]{Lemma}
\newtheorem*{problem}{Problem}

\begin{document}

\title{Points defining triangles with distinct circumradii} 
\author{Leonardo Mart\'inez$^1$}
\thanks{$^1$ Research supported by CONACyT grant 277462.}

\author{Edgardo Rold\'an-Pensado$^2$}
\thanks{$^2$ Research supported by CONACyT project 166306 and ERC Advanced Grant 267165.}

\address{Instituto de Matem\'aticas, Unidad Juriquilla\\
Universidad Nacional Aut\'onoma de M\'exico}
\email[Leonardo Mart\'inez]{\href{mailto:leomtz@im.unam.mx}{leomtz@im.unam.mx}}
\email[Edgardo Rold\'an-Pensado]{\href{mailto:e.roldan@im.unam.mx}{e.roldan@im.unam.mx}}



\begin{abstract}
Paul Erd\H os asked if, among sufficiently many points in general position, there are always $k$ points such that all the circles through $3$ of these $k$ points have distinct radii. He later proved that this is indeed the case. However, he overlooked a non-trivial case in his proof. In this note we deal with this case using B\'ezout's Theorem on the number of intersection points of two curves and obtain a polynomial bound for the needed number of points.
\end{abstract}

\maketitle

\section{Introduction}

In 1975 \cite{Erd1975}, inspired by the observations from Esther Szekeres and his results with George Szekeres, Paul Erd\H os posed the following problem:

\begin{problem}
Is it true that for every $k$ there is an $n_k$ such that if there are given $n_k$ points in the plane in general position (i.e. no three on a line no four on a circle) one can always find $k$ of them so that all the $\binom{k}{3}$ triples determine circles of distinct radii?
\end{problem}

This problem is similar to the Erd\H os-Szekeres Theorems \cite{ES1935}. As is the case with these theorems, the existence of $n_k$ can be established using Ramsey Theory if the existence of $n_6$ can be verified. To do this, consider the Ramsey number $R(k,n_6;6)$ and take this number of points in the plane in general position. Color a $6$-tuple of points green if all $20$ triangles determined by these points have distinct circumradii and red otherwise. Then Ramsey's Theorem \cite{Ram1930} gives us either a subset with $n_6$ elements such that all $6$-tuples are red or a subset with $k$ elements such that all $6$-tuples are green. The first option is impossible by the definition of $n_6$ and the second one implies that the $k$ points determine triangles with distinct circumradii. However, establishing the existence of $n_6$ is not completely trivial and the bound obtained from this method is an exponential tower.

Three years later, in 1978, Erd\H os published a paper \cite{Erd1978} where he claimed a positive answer to the question with $n_k\le 2\binom{k-1}{2}\binom{k-1}{3}+k$. However, he inadvertently left out a non-trivial case for which his method does not work.
It seems that Erd\H os remained unaware of this and even restated the result in 1985 \cite{Erd1985} giving partial credit to E. Straus.

In this note we address this issue and give a polynomial bound for $n_k$.

\begin{theorem}\label{thm:main}
Let $k$ be a positive integer and let $n_k$ be the smallest integer such that the following holds:
For any $n_k$ points in the plane in general position (i.e. no four on a line or circle) there are $k$ of them so that all their triples determine circles of distinct radii.
Then $n_k=O(k^9)$.
\end{theorem}

We prove this in section \ref{sec:general}. Note that we redefined $n_k$ and changed the general position condition to what we think is a more suitable one, since a line is just a circle of infinite radius.
In section \ref{sec:erd_proof} we examine Erd\H os' argument. The proof of Theorem \ref{thm:main} is based on a very similar method but slightly more involved.

In section \ref{sec:small} we analyze $n_k$ for $k=4,5$ and give explicit bounds for them. These are used in the proof of \ref{thm:main}. To be precise we prove the following.

\begin{theorem}\label{thm:small}
The first two non-trivial values of $n_k$ satisfy $n_4\le 9$ and $n_5\le 37$.
\end{theorem}

Before we continue, we fix some notation to be used throughout the paper. If $\F$ is a set, $\binom{\F}{m}$ denotes the set of unordered $m$-tulpes of $\F$. Given points $A,B,C$ in the plane, $\abs{AB}$ is the length of the segment $AB$, $\abs{ABC}$ is the area of the triangle $ABC$ and $R(ABC)$ is the circumradius of the triangle $ABC$.

\section{Erd\H os' argument}\label{sec:erd_proof}

Now we look at the argument from \cite{Erd1978}, in which Erd\H os claims $n_k\le k+\binom{k-1}{2}\binom{k-1}{3}$.

\begin{proof}[Erd\H os' argument]
We start with a set $\F$ of $n\ge k+\binom{k-1}{2}\binom{k-1}{3}$ points in the plane in general position and chooses a maximal set $\G\subset \F$ so that all the triples in $\binom{\G}{3}$ determine circles of distinct radii. Let $l=\abs{\G}$ and assume that $l<k$. Denote the circumradii by $r_1,\dots,r_{\binom{l}{3}}$

Since $\G$ is maximal, every point of $\F\setminus \G$ must lie in a circle of radius $r_i$ through two points of $\G$. But because the points are in general position, there can be at most one point in such a circle. Note that there are $\binom{l}{3}$ radii, $\binom{l}{2}$ pairs of points, and at most two circles of a certain radius through two points.
Therefore $n-l\le 2\binom{l}{3}\binom{l}{2}$, which is a contradiction.
\end{proof}

There is a problem here. Namely, that a point of $\F\setminus \G$ must not necessarily be in one of the circles described above. For example, a point $X\in\F\setminus\G$ could satisfy $R(ABX)=R(CDX)$, with $A,B,C,D\in\F$, without being in any of the circles described above and not contradict the maximality of $\G$. To fix the gap in this argument, we give a polynomial bound for the number of such points $X$.

\section{Small cases}\label{sec:small}

Here we show that $n_k$ is bounded for $k\le 5$, as we need this for the general case. The proofs are mostly combinatorial. In fact the only geometric property we use is the following: if $3$ triangles have the same circumradius and share an edge, then $4$ of their vertices lie on a circle.

\begin{proof}[Proof of first part of Theorem \ref{thm:small}]
Assume we have a set $\F$ of $9$ points in general position and among every $4$ of them there are two triangles with equal circumradius. Then to every $\G\subset \F$ consisting of $4$ points we can assign two pairs of points, say $f(\G)=\{A,B\}\subset \G$ and $g(\G)=\{C,D\}\subset \G$, such that $R(ABC)=R(ABD)$. These are functions $f,g:\binom{\F}4\to\binom{\F}2$.
Here $f(\G)$ are the points that form the common base of the triangles with equal circumradius in $\G$ and $g(\G)$ are the other two vertices.

There are $126$ sets of $4$ points and only $36$ pairs of points, therefore there is a pair of points, say $\{A,B\}$, such that $f^{-1}(\{A,B\})$ has at least $4$ elements.
Since there are only $7$ points in $\F\setminus\{A,B\}$, there are $\G_1,\G_2\in f^{-1}(\{A,B\})$ such that $g(\G_1)\cap g(\G_2)\neq\emptyset$. Assume that $\G_1=\{A,B,C,D\}$ and $\G_2=\{A,B,C,E\}$, then $R(ABC)=R(ABD)=R(ABE)$. But this implies that $4$ points lie in some circle, contradicting the general position assumption.
\end{proof}

\begin{proof}[Proof of second part of Theorem \ref{thm:small}]
Assume we have a set $\F$ of $37$ points and in every set $\G\subset \F$ of $5$ points there are two triangles with equal circumradius. These two triangles must have a vertex in common, let $f(\G)=A$ be this vertex and let $g(\G)=\{\{B,C\},\{D,E\}\}\subset\binom{\G}{2}$ be the other vertices of the triangles so that $R(ABC)=R(ADE)$.

Since there are only $37$ points, there is a point $A$ assigned by $f$ to $\frac{1}{37}\binom {37}5$ of the $5$-tuples. These $5$-tuples are mapped by $g$ into a total of $\frac{2}{37}\binom {37}5$ pairs of points in $\G$ so there is a pair, say $\{B,C\}$, in $\cei{\frac{2}{37}\binom {37}5/\binom {37}2}=36$ of them.

Now consider the $5$-tuples $\G$ such that $\{B,C\}\in g^{-1}(\G)$, and for each of these take the other pair $\{D_\G,E_\G\}\in g^{-1}(\G)$.
Note that $R(AD_\G E_\G)=R(ABC)$ for all such $\G$, giving a total of $37$ triangles with equal circumradius and a common vertex $A$.
Since there are only $36$ points in $\F\setminus\{A\}$, there must be another point belonging to $3$ of these triangles. But these three triangles have an edge in common, therefore $4$ of their vertices lie in a circle contradicting the general position assumption.
\end{proof}

\section{General case}\label{sec:general}

Here we prove Theorem \ref{thm:main}, but we need some additional definitions and lemmas.

Consider $\{A,B\}$ and $\{C,D\}$ two different pairs of points on the plane. We are interested in the locus of the points $X$ such that $R(ABX)=R(CDX)$, which we denote by $\C(AB,CD)$.
Since the circumradius of a triangle satisfies
\[
R(ABX)= \frac{\abs{AX}\abs{BX}\abs{AB}}{4\abs{ABX}},
\]
$\C(AB,CD)$ is the algebraic curve of degree at most $6$ defined by the zero set of
\[
\abs{AX}^2\abs{BX}^2\abs{AB}^2\abs{CDX}^2-\abs{CX}^2\abs{DX}^2\abs{CD}^2\abs{ABX}^2.
\]

Now assume we have a set $\F$ of $n$ points in general position, let $\G$ be a maximal subset of those points such that all its triples determine circles of distinct radii and set $l=\abs{\G}$. Since $\G$ is maximal, each of the remaining $n-l$ points must lie on one of the following curves.
\begin{enumerate}
\item A circle through $2$ points of $\G$ with radius $r_i$ for some $i$.
\item The curve $\C(AB,CD)$ with $\{A,B\}$ and $\{C,D\}$ distinct pairs of points of $\G$.
\end{enumerate}

Our goal is to bound the number of points in $\F\setminus\G$, but first we address a particular case of our main theorem.

\begin{lemma}\label{lem:irred}
Let $\D$ be an irreducible algebraic curve of degree at most $6$.
Then for every integer $k$ there exists an integer $m_k=O(k^5)$ such that the following holds:
every set $\F\subset\D$ with $m_k$ points in general position contains a subset $\G$ with $k$ points such that all its triples determine circles of distinct radii.
\end{lemma}
\begin{proof}
Take $m$ points in general position on $\D$, let $\G$ be a maximal set of these points such that all its triples determine circles of distinct radii and let $l=\abs{\G}$. By Theorem \ref{thm:small} we may assume $l\geq 5$.

Each of the remaining $m-l$ points must lie on one of the two curves described above.

In Case (1), we use Erd\H os' argument: there are $\binom{l}{3}$ possible radii, $\binom{l}{2}$ possible pairs and since the points are in general position, there are at most two points for each radius and each pair. This bounds above the number of points in Case (1) by $2\binom{l}{2}\binom{l}{3}$.

In Case $(2)$, there are $\frac 12\binom{l}{2}(\binom{l}{2}-1)$ ways to choose the pairs $\{A,B\}$ and $\{C,D\}$. Consider the curve $\C(AB,CD)$, by B\'ezout's Theorem \cite{Coo1931}, either $\D$ is an irreducible component of $\C(AB,CD)$ or $\D\cap\C(AB,CD)$ contains at most $36$ points.
But if $\D\subset\C(AB,CD)$, then $\G=\{A,B\}\cup\{C,D\}$ because any other point of $\G$ would lie on $\C(AB,CD)$ and thus repeat a radius. This contradicts $l\geq 5$. Therefore $\D\cap\C(AB,CD)$ has at most $36$ points. This bounds above the number of points in Case (2) by $\frac {36}2\binom{l}{2}(\binom{l}{2}-1)$.

So the number of points in $\F\setminus\G$ is
\[m-l\le 2\binom{l}{2}\binom{l}{3} + 36\binom{\binom{l}{2}}{2},\]

from which the desired bound follows.

\end{proof}

Now the proof of Theorem \ref{thm:main} is straightforward.

\begin{proof}[Proof of Theorem \ref{thm:main}]
Once again, assume $\F$ has $n$ points and $l$ is the size of the maximal $\G\subset\F$ with all its triples having distinct circumradii. For the remaining $n-l$ points we have the same two cases.

We can bound the number of points in Case (1) by $2\binom{l}{2}\binom{l}{3}$. For case (2) we use Lemma \ref{lem:irred} to obtain a bound of $\frac 12\binom{l}{2}(\binom{l}{2}-1)m_l$. This gives
\[n-l\le 2\binom{l}{2}\binom{l}{3} + \binom{\binom{l}{2}}{2}m_l,\]
which implies that $n_k=O(k^9)$.
\end{proof}

\end{document}